\documentclass[10pt,noinfoline
]{article}

\RequirePackage[OT1]{fontenc}
\RequirePackage[%lnms,
amsthm,amsmath%,natbib
]{imsart}

\usepackage{graphicx}
\usepackage{latexsym,amsmath}
\usepackage{amsmath,amsthm,amscd}
\usepackage{amsfonts}
\usepackage[psamsfonts]{amssymb}

\usepackage{mathabx} 

\usepackage{enumerate}

\usepackage{url}

\usepackage{tocvsec2}

\usepackage{epstopdf}
\usepackage{wrapfig}
\usepackage{float}

\usepackage{calrsfs}

\usepackage{subfigure}
\usepackage[small]{caption}
\usepackage{hyperref}

\startlocaldefs
\theoremstyle{plain} 
\newtheorem{theorem}{Theorem}%[section]
\newtheorem{corollary}[theorem]{Corollary}
\newtheorem*{cor2a}{Corollary~2a}

\theoremstyle{definition} 
\newtheorem{definition}[theorem]{Definition}
\newtheorem*{definition*}{Definition}
\theoremstyle{definition} 
\newtheorem{ex}[theorem]{Example}
\newtheorem*{ex*}{Example}
\theoremstyle{remark} 

\theoremstyle{remark} 

\newtheorem*{remark*}{Remark}

\makeatletter
\def\subsubsubsection{\@startsection{subsubsubsection}{4}{\z@}{-3.25ex plus -1ex minus -.2ex}{1.5ex plus .2ex}{\normalsize}}
\makeatother                                   
 
\newcommand{\beqa}{\begin{eqnarray}}
\newcommand{\eeqa}{\end{eqnarray}}
 
\newcommand{\bseq}{\begin{subequations}}
 
\newcommand{\eseq}{\end{subequations}}

\newcommand{\epi}{{\,\operatorname{epi}}}
\newcommand{\dom}{{\,\operatorname{dom}}}

\newcommand{\card}{\operatorname{card}}

\newcommand{\si}{\sigma}

\newcommand{\la}{\lambda}
\newcommand{\Nu}{\mathrm{N}}

\newcommand{\de}{\delta}

\renewcommand{\Psi}{\overline{\Phi}}

\newcommand{\E}{\operatorname{\mathsf{E}}}

\newcommand{\R}{\mathbb{R}}

\newcommand{\CC}{\mathcal{C}}

\newcommand{\vp}{\varepsilon}

\newcommand{\tf}{{\tilde{f}}}

\renewcommand{\le}{\leqslant}
\renewcommand{\ge}{\geqslant}

\endlocaldefs

\begin{document}

\begin{frontmatter}

\title{A necessary and sufficient condition on the stability of the infimum of convex functions}
\runtitle{Infimum of convex functions}
%\date{\today}

% \author{\fnms{First}  \snm{Author}\corref{}\thanksref{t2}\ead[label=e1]{first@somewhere.com}},
%  \author{\fnms{Second} \snm{Author}\ead[label=e2]{second@somewhere.com}}
%  \and
%  \author{\fnms{Third}  \snm{Author}%
%  \ead[label=e3]{third@somewhere.com}%
%  \ead[label=u1,url]{http://www.foo.com}}
%
%  \thankstext{t2}{Footnote to the first author with the `thankstext' command.}

\begin{aug}
\author{\fnms{Iosif} \snm{Pinelis}\thanksref{t2}\ead[label=e1]{ipinelis@mtu.edu}}
  \thankstext{t2}{Supported by NSA grant H98230-12-1-0237}
\runauthor{Iosif Pinelis}

%\affiliation{Michigan Technological University}

\address{Department of Mathematical Sciences\\
Michigan Technological University\\
Houghton, Michigan 49931, USA\\
E-mail: \printead[ipinelis@mtu.edu]{e1}}
\end{aug}

\begin{abstract}
Let us say that 
a convex  
function $f\colon C\to[-\infty,\infty]$ on a convex set $C\subseteq\R$ is infimum-stable if, for any sequence $(f_n)$ of convex functions  
$f_n\colon C\to[-\infty,\infty]$ converging to $f$ pointwise, one has  
$\inf\limits_C f_n\to\inf\limits_C f$. 
A simple necessary and sufficient condition for a convex function  
to be infimum-stable is given. 
The same condition remains necessary and sufficient if one uses Moore--Smith nets $(f_\nu)$ in place of sequences $(f_n)$.  
This note is motivated by certain applications to stability of measures of risk/inequality in finance/economics. 
\end{abstract}

%\subjclass[2000]{60E15, 62G10, 62G15, 60G50, 62G35}
% 62G10    	Hypothesis testing
%  62G15    	Tolerance and confidence regions
%  60G50    	Sums of independent random variables; random walks
%   62G35    	Robustness
  
%
%\keywords{probability inequalities; Rade\-macher random variables; sums of independent random variables; Student's test; self-normalized sums}

\setattribute{keyword}{AMS}{AMS 2010 subject classifications:}


% 	26A48   	Monotonic functions, generalizations
%		26A51   	Convexity, generalizations

% 	26A51   	Convexity, generalizations	
% 	
% 	 	90C25   	Convex programming	
%
% 	49J45   	Methods involving semicontinuity and convergence; relaxation
%49J52   	Nonsmooth analysis [See also 46G05, 58C50, 90C56]
%49K05   	Free problems in one independent variable
% 	49K30   	Optimal solutions belonging to restricted classes

%60E15   	Inequalities; stochastic orderings
%62E17   	Approximations to distributions (nonasymptotic)
%62H10   	Distribution of statistics
%62H15   	Hypothesis testing [multivar]
%62F03   	Hypothesis testing [param]
%62G10   	Hypothesis testing [nonpar]


\end{frontmatter}

\settocdepth{chapter}

\tableofcontents 
%%%%%%%%%%%%%%%%%{\small\tableofcontents} 

\settocdepth{subsubsection}

\theoremstyle{plain} 
%\newtheorem{theorem}{Theorem}[section]
%\newtheorem{corollary}[theorem]{Corollary}
%\newtheorem*{main}{Main~Theorem}
%\newtheorem{lemma}{Lemma}[subsection]
%\newtheorem{proposition}[theorem]{Proposition}
%\newtheorem{conjecture}{Conjecture}
%\theoremstyle{definition} 
%\newtheorem{definition}[theorem]{Definition}
%\theoremstyle{definition} 
%\newtheorem{ex}{Example}
%\theoremstyle{remark} 
%\newtheorem{exer}{Exercise}
%\theoremstyle{remark} 
%\newtheorem{remark}[theorem]{Remark}b
%\newtheorem*{remark*}{Remark}
%\numberwithin{equation}{section}

%\eject

\section{Summary and discussion}\label{intro} 

In this paper, the general notion of the convexity of functions is assumed, following \cite{rocka}. Namely, let $C$ be any convex subset of $\R$, and then 
take any function $f\colon C\to[-\infty,\infty]$. %; that is, $C$ is an interval a 
The function $f$ is called convex if its epigraph 
\begin{equation*}
	\epi f:=\{(x,\tau)\in C\times\R\colon\tau\ge f(x)\} 
\end{equation*}
is a convex set. 
The effective domain of 
$f$ %a function $f\colon C\to[-\infty,\infty]$ 
is 
\begin{equation*}
	\dom f:=\{x\in C\colon f(x)<\infty\}. 
\end{equation*}

The convexity of $f$ %a function $f\colon C\to[-\infty,\infty]$ 
can also be expressed by the usual inequality  
\begin{equation}\label{eq:conv}
	f\big((1-\la)x+\la y\big)\le(1-\la)f(x)+\la f(y)  
\end{equation}
for all $x\in C$, $y\in C$, and $\la\in(0,1)$ -- but with the exception of the case when $\{f(x),f(y)\}=\{\infty,-\infty\}$, that is, when $f(x)$ and $f(y)$ are infinite values of opposite signs; cf.\ e.g.\ \cite[Theorem~4.1]{rocka}. 
Equivalently, one may require \eqref{eq:conv} only for all $x$ and $y$ in $\dom f$ \big(and still for all $\la\in(0,1)$\big).

Generally, whenever possible, let us allow variables in expressions to take infinite values. At that, the corresponding values of the expressions are defined ``by continuity'', as usual; e.g., the value of the expression $\sqrt{\rho-\si-1}$ at $(\rho,\si)=(\infty,-\infty)$ should be understood as $\lim\limits_{\substack{\rho\to\infty,\\ \si\to-\infty}}\sqrt{\rho-\si-1}=\infty$. 

%The effective domain of a function $f\colon\R\to[-\infty,\infty]$ is 
%\begin{equation}
%	\dom f:=\{x\in\R\colon f(x)<\infty\}. 
%\end{equation}

%Let us also say that $f$ % a function $f\colon C\to[-\infty,\infty]$ 
%is monotonic if $f$ is nondecreasing or nonincreasing. 
%As usual, the symbol $\card$ denotes the cardinality of the given set. 

%The following is the main definition in this note.
One can now give 

\begin{definition}\label{def:inf-stable}
%Let us say that the %convex 
Say that the 
function $f%\colon C\to[-\infty,\infty]
$ is \emph{(sequentially) infimum-stable} if, for any sequence $(f_n)$ of convex functions %\break 
$f_n\colon C\to[-\infty,\infty]$ converging to $f$ pointwise, one has  
$$\inf\limits_C f_n\to\inf\limits_C f.$$ 
\end{definition}

%\noindent 
As usual, assume the conventions $\inf\emptyset:=\infty$ and $\sup\emptyset:=-\infty$. 
Also as usual, let the symbol $\card$ denote the cardinality of the given set. 

If $\card C\le1$ then, obviously, $f$ is necessarily infimum-stable. 
%any function $f\colon C\to[-\infty,\infty]$ is infimum-stable. 
To avoid this triviality, let us further assume that $\card C>1$, so that $C$ is an interval with endpoints 
\begin{equation*}
	c_{\min}:=\inf C\quad\text{and}\quad c_{\max}:=\sup C, 
\end{equation*}
and at that 
\begin{equation}\label{eq:cmin<cmax}
	-\infty\le c_{\min}<c_{\max}\le\infty. 
\end{equation}

Let us also say that $f$ % a function $f\colon C\to[-\infty,\infty]$ 
is monotonic if $f$ is nondecreasing or nonincreasing. 

\begin{center}***\end{center}

We shall provide a simple necessary and sufficient condition for a convex function on $C$ 
%$f\colon C\to[-\infty,\infty]$ 
to be infimum-stable. The form of this condition depends on 
which (if any) 
%whether either one or both 
of the endpoints $c_{\min}$ and $c_{\max}$ of the interval $C$ are infinite, so that there are four possible cases here. 
The main case, when $c_{\min}=-\infty$ and $c_{\max}=\infty$, is described by 

\begin{theorem}\label{th:}
Suppose that $C=\R$ and the function $f$ is convex. Then the following two conditions are equivalent to each other: 
\begin{enumerate}[(I)]
	\item $f$ is infimum-stable; 
	\item either (i) $\card\dom f>1$ and at that $f$ is not monotonic or (ii) $\inf\limits_\R f=-\infty$. 	
%	\item (i) $\card\dom f>1$ and (ii) $\inf\limits_\R f=-\infty$ or $f$ is not monotonic. 	
\end{enumerate} 
\end{theorem}

The proofs are given in Section~\ref{proofs}. 

The remaining three cases concerning the finiteness of $c_{\min}$ and/or $c_{\max}$ are presented in the following three corollaries. 

\begin{corollary}\label{cor:fin,fin}
Suppose that $-\infty<c_{\min}<c_{\max}<\infty$ and the function $f$ is convex. Then the following two conditions are equivalent to each other: 
\begin{enumerate}[(I)]
	\item $f$ is infimum-stable; 
	\item either (i) $\card\dom f>1$ or (ii) $\inf\limits_C f=-\infty$. 	 	
\end{enumerate} 
\end{corollary}

\begin{corollary}\label{cor:-inf,fin}
Suppose that $-\infty=c_{\min}<c_{\max}<\infty$ and the function $f$ is convex. Then the following two conditions are equivalent to each other: 
\begin{enumerate}[(I)]
	\item $f$ is infimum-stable; 
	\item either (i) $\card\dom f>1$ and at that $f$ is not nondecreasing or (ii) $\inf\limits_C f=-\infty$. 	 	
\end{enumerate} 
\end{corollary}

\begin{corollary}\label{cor:fin,inf}
Suppose that $-\infty<c_{\min}<c_{\max}=\infty$ and the function $f$ is convex. Then the following two conditions are equivalent to each other: 
\begin{enumerate}[(I)]
	\item $f$ is infimum-stable; 
	\item either (i) $\card\dom f>1$ and at that $f$ is not nonincreasing or (ii) $\inf\limits_C f=-\infty$. 	 	
\end{enumerate} 
\end{corollary}

When, in addition to the conditions in Corollary~\ref{cor:fin,fin} that the set $C$ be bounded and the function $f$ on $C$ be convex, it is known that $f>-\infty$ on $C$, then the nececcary and sufficient condition for $f$ to be infimum-stable can be simplified: 

\begin{cor2a}\label{cor:2a}
Suppose that $-\infty<c_{\min}<c_{\max}<\infty$, the function $f$ is convex, and $f>-\infty$ on $C$. Then $f$ is infimum-stable if and only if $\card\dom f>1$.  
\end{cor2a}

The above results are significantly easier to prove, if one additionally assumes that the functions $f$ and $f_n$ take only real values. Indeed, 
if $f$ is not monotonic on $C$, then $f(v)<f(u)\wedge f(w)$ for some $u$, $v$, and $w$ in $C$ such that $u<v<w$ and hence the convexity of $f$ yields $\inf\limits_C f=\inf\limits_{C\cap[u,w]} f$. So, for all large enough $n$ one has  $f_n(v)<f_n(u)\wedge f_n(w)$ and hence $\inf\limits_C f_n=\inf\limits_{C\cap[u,w]} f_n$. By \cite[Theorem~10.6]{rocka}, the real-valued $f_n$'s are uniformly bounded and equi-continuous on $C\cap[u,w]$. So, by the Arzel\`a--Ascoli theorem, $f_n\to f$ uniformly on $C\cap[u,w]$. 
Thus, $\inf\limits_C f_n=\inf\limits_{C\cap[u,w]} f_n\longrightarrow\inf\limits_{C\cap[u,w]} f=\inf\limits_C f$. 

Alternatively, if $C=\R$ and again $f$ and $f_n$ take only real values, 
then the convexity condition implies that these functions are continuous and 
%, provided the condition $\card\dom f>1$, 
attain a minimum value, which yields another kind of short proof \cite{wong}. 

However, it is oftentimes useful to allow $f$ and $f_n$ to take the value $\infty$. In particular, one may restate any constrained minimization problem, of minimizing a convex function $f$ over a convex set $C\subset\R$, as an unconstrained one, of minimizing the corresponding convex extension $\tf$ of $f$ over the entire set $\R$, where 
\begin{equation}\label{eq:tf}
\tf:=
\left\{
\begin{alignedat}{2}
&f&&\text{ on } C, \\ 
&\infty&&\text{ on } \R\setminus C. 
\end{alignedat}
\right.
\end{equation}
In fact, such extensions will be used in Section~\ref{proofs}, in the proofs of Corollaries~\ref{cor:fin,fin}, \ref{cor:-inf,fin}, and \ref{cor:fin,inf}. 

Also, the value $\infty$ of convex functions arises naturally, for instance, in applications in probability, where the moment generating function of a random variable $X$ (defined by the formula $M_X(t):=\E e^{tX}$ for $t\in\R$) may assume the value $\infty$; clearly, this function and even its logarithm are convex. In fact, the present note was motivated by certain applications to measures of risk/inequality in finance/economics. 

As for allowing $f$ and $f_n$ to also take the value $-\infty$, this provides a further generalization, without significantly complicating the proofs. 

Another application of Theorem~\ref{th:} is to the Legendre--Fenchel transform $f^*\colon\R\to[-\infty,\infty]$ of $f$, defined by the formula 
\begin{equation*}
	f^*(t):=\sup_{x\in\R}[tx-f(x)]
\end{equation*}
for $t\in\R$. This application is quite straightforward: 

\begin{corollary}\label{cor:f^*}
Suppose that $C=\R$ and the function $f$ is convex and not monotonic, with $\card\dom f>1$. Then the function $f^*$ is continuous at $0$. 
\end{corollary}

The need for results given above appears to arise quite naturally. However, I have been able to find only similar, but not quite the same, results in the existing literature. 

%Even though the need for results given above appears to arise quite naturally, I have not been able to find them in the existing literature.
%The two closest kinds of results appear to be as follows. 

One series of papers (see e.g.\ \cite{bonnans-shapiro00}) concerns a perturbed function $f\colon X\times U\to\R$, where $U$ is the set of values of a perturbation parameter. Among other conditions, $f(x,u)$ is assumed there to be at least continuous, jointly in $x\in X$ and $u\in U$, with the conclusion that the value function $u\mapsto\inf\limits_{x\in X}f(x,u)$ is continuous. 

One should also mention \cite[Theorem~1.2]{komuro} concerning, essentially, the case when $f$ and $f_n$ are \emph{real-valued} convex functions defined on a \emph{compact} convex subset of $\R^d$. 
%
%Yet, another series of results potentially the closest to ours, are based on considerations of another kind of convergence of $f_n$ to $f$, usually called the $\tau$- or $\tau_M$-convergence (after Mosco \cite{mosco}), or its generalization to infinite dimensions, the $\tau_{aw}$

Another series of papers deals  
%, instead of the pointwise convergence of $f_n$ to $f$, 
with other kinds of convergence of $f_n$ to $f$, logically more complicated than the pointwise convergence --- mostly in the more general setting, when the functions are defined on a reflexive Banach space $E$. For instance, the $\tau$-convergence $f_n\overset\tau\to f$ on $E$ can be stated for $E=\R$ as follows (see e.g.\ \cite[Lemma~1]{salin-wets77}): 
for any $x\in \R$ and any sequence $(x_n)$ in $\R$ such that $x_n\to x$ one has $\liminf_n f_n(x_n)\ge f(x)$, and every $x\in \R$ is the limit of a sequence $(x_n)$ in $\R$ such that %$x_n\to x$ and 
$\limsup_n f_n(x_n)\le f(x)$ \big(and hence $\lim_n f_n(x_n)=f(x)$ for such a sequence $(x_n)$\big). 
Also, in contrast with our results, the convex functions $f$ and $f_n$ are assumed in \cite{salin-wets77} to be lower-semicontinuous (l.s.c.) and proper. 
As shown in \cite{salin-wets77}, the $\tau$-convergence of $f_n$ to $f$ is equivalent to the pointwise convergence --- provided the additional condition that the $f_n$'s are \emph{equi}-lower-semicontinuous; however, without the latter condition, the relation between the $\tau$-convergence  and the pointwise convergence appears unclear. 
A potentially very interesting result is \cite[Corollary~2C]{salin-wets77}, stating the following: if $f_n$ and $f$ are l.s.c.\ proper convex functions on $\R^d$, $f_n\to f$ pointwise, and the interior of $\dom f$ is nonempty, then $f_n\overset\tau\to f$. 
Using then \cite[Theorem~3.7]{beer-lucc}, one could deduce the implication (II)$\implies$(I) of our Theorem~\ref{th:} --- in the case when the functions $f$ and $f_n$ are l.s.c.\ and proper. However, there appears to be a gap in the proof of the mentioned \cite[Corollary~2C]{salin-wets77}. Namely, part (iii) of \cite[Lemma~2]{salin-wets77} is incorrect in general, without additional assumptions. Indeed, consider 

\begin{ex}\label{ex:}
Let functions $f$ and $f_n$ on $\R^2$ be defined by the conditions that $f(0,0)=0$, $f(x,y)=\infty$ for $(x,y)\in\R^2\setminus\{(0,0)\}$, and $f_n(x,y)=n|y+nx|$ for all natural $n$ and all $(x,y)\in\R^2$. Then the functions $f$ and 
$f_n$ are convex, l.s.c., and proper, and 
$f_n\to f$ pointwise. 
Also, the set $K:=[-1,0]\times\{1\}$ is a compact subset of the set $\R^2\setminus\{0\}$, which latter coincides with the complement to $\R^2$ of the closure of the set $\dom f=\{0\}$. 
%$\R^2\setminus\dom f=\R^2\setminus\{0\}$. 
However, $\min\limits_K f_n=0$ for all natural $n$, and hence $f_n$ does not go to $\infty$ uniformly on $K$. 
\end{ex}

Such an example would be impossible for functions $f$ and $f_n$ defined on $\R$, as can be seen from the proof of Theorem~\ref{th:} below. 
So, Example~\ref{ex:} suggests that one would have to overcome some additional difficulties to extend our results to convex functions defined on a linear space of dimension greater than $1$.
 
\begin{center}***\end{center}

Instead of sequences $(f_n)$, one can, more generally, deal here with (Moore--Smith) nets $(f_\nu)_{\nu\in\Nu}$ of convex functions $f_\nu$, where $\Nu$ is an arbitrary directed set; see e.g.\ \cite[Chapter~2]{kelley55} concerning the relevant terms of general topology. 
Equivalently, one can define an ostensibly stronger notion of infimum-stability in terms of a topology, say $\pi_{\CC(C)}$, on the set $\CC(C)$ of all convex functions $f\colon C\to[-\infty,\infty]$ on a convex set $C\subseteq\R$. 
Namely, $\pi_{\CC(C)}$ should be the topology  induced on $\CC(C)$ by the Tychonoff product topology $\pi_C$ on the set $[-\infty,\infty]^C$ of all functions $f\colon C\to[-\infty,\infty]$ on $C$, so that a subbase of the topology $\pi_{\CC(C)}$ consists of all sets of the form $F_{x,I}:=\{f\in\CC(C)\colon f(x)\in I\}$, where $x$ is any point in $C$ and $I$ is any interval of one of the following three forms: $(a,b)$, $[-\infty,b)$, or $(a,\infty]$, for arbitrary real $a$ and $b$. 

%the interval $[-\infty,\infty]$. 
%Then one can give 

\begin{definition}\label{def:top-inf-stable}
%Let us say that the %convex 
A
function $f\in\CC(C)$ is \emph{topologically infimum-stable} if the mapping %\break 
$\CC(C)\ni g\mapsto\inf g$ is continuous at ``point'' $f$ with respect to the topology $\pi_{\CC(C)}$. 
\end{definition}

Closely following the lines of the proofs of 
Theorem~\ref{th:} and Corollaries~\ref{cor:fin,fin}, \ref{cor:-inf,fin}, \ref{cor:fin,inf},  and 2a, stated above, and substituting terms such as ``net'' and ``subnet'' for ``sequence'' and ``subsequence'', one sees that all these statements hold with nets $(f_\nu)_{\nu\in\Nu}$ in place of sequences $(f_n)$; at that, the term ``eventually'' should be understood in a standard way, as e.g.\ in \cite[page~65]{kelley55}: a property $P_\nu$ holds eventually if there is some $\mu\in\Nu$ such that $P_\nu$ holds for each $\nu\in\Nu$ satisfying the condition $\nu\ge\mu$. 
%
%that is, as ``for some $\mu\in\Nu$ and all $\nu\in\Nu$ such that $\nu\ge\mu$''. 
Thus, one obtains 

\begin{theorem}\label{th:equiv}  
For any convex set $C\subseteq\R$, a function $f\in\CC(C)$ is topologically infimum-stable if and only if it is sequentially infimum-stable. 
\end{theorem}

\section{Proofs}\label{proofs} 

\begin{proof}[Proof of Theorem~\ref{th:}]\ Let us first consider the implication 

(II)$\implies$(I).\quad Assume that the function $f\colon\R\to[-\infty,\infty]$ is convex and satisfies the condition (II), and take any sequence $(f_n)$ of convex functions $f_n\colon\R\to[-\infty,\infty]$ converging to $f$ pointwise. 
%Then it is easy to see that $f$ is convex. 
Let 
\begin{equation}\label{eq:m_n,m}
	m:=\inf\limits_\R f\quad\text{and}\quad m_n:=\inf\limits_\R f_n. 
\end{equation}	
We need to show that $m_n\to m$. If this is not true, then there is a sequence $(n_k)$ of natural numbers such that $n_k\to\infty$ and $m_{n_k}$ converges to a point in $[-\infty,\infty]\setminus\{m\}$ as $k\to\infty$. So, without loss of generality (w.l.o.g.) there is a limit 
\begin{equation}\label{eq:m_infty}
	m_\infty:=\lim_n m_n. 
\end{equation}

Take any $c\in(m,\infty)$. Then, by the definition of $m$, 
there is some $x_c\in\R$ such that $f(x_c)<c$. By the pointwise convergence of $f_n$ to $f$, it follows that eventually (that is, for all large enough $n$) $f_n(x_c)<c$ and hence $m_\infty\le\limsup_n f_n(x_c)\le c$. 
Thus (in fact, without any conditions on $f$), $m_\infty\le m$. 
So, w.l.o.g.\ 
\begin{equation}\label{eq:limsup<m}
	m_\infty
%	:=
%	\limsup_n m_n
	<m. 
\end{equation}
In particular, it follows that $m>-\infty$. 
%So, assume in the rest of the proof of the implication (II)$\implies$(I) that $m>-\infty$. 
On the other hand, the condition $\card\dom f>1$ shows that $m<\infty$. Thus, 
\begin{equation}\label{eq:m in R}
	m\in\R.   
\end{equation}
So, by vertical shifting, w.l.o.g.\ one may assume that 
\begin{equation}\label{eq:m=1}
	m=1;     
\end{equation}
we shall actually use this latter assumption only where convenient.  

Thus, the alternative (i) of the condition (II) takes place: $\card\dom f>1$ and $f$ is not monotonic. 
%
%%and the convexity of $f$ can be expressed by the formula 
%%\begin{equation}\label{eq:conv}
%%	f\big((1-\la)x+\la y\big)\le(1-\la)f(x)+\la f(y)\quad\text{for all}\quad 
%%	x\in\R,\ y\in\R,\ \la\in(0,1). 
%%\end{equation}
%%The assumption 
%By \eqref{eq:m in R} and the condition (II)(ii), $f$ is not monotonic. 
So, there are some real numbers $x_1$ and $x_2$ such that $x_1<x_2$ and $f(x_1)<f(x_2)$, which in particular implies $f(x_1)<\infty$. 
Hence, by \eqref{eq:conv}, for any $x\in(x_2,\infty)$ with $f(x)<\infty$ one has  $f(x_2)\le\frac{x-x_2}{x-x_1}\,f(x_1)+\frac{x_2-x_1}{x-x_1}\,f(x)$ or, equivalently, 
\begin{equation}\label{eq:f(x)ge}
f(x)\ge
%\frac{x-x_1}{x_2-x_1}\,f(x_2)-\frac{x-x_2}{x_2-x_1}\,f(x_1)
%=
f(x_2)+\frac{x-x_2}{x_2-x_1}\big(f(x_2)-f(x_1)\big),  	
\end{equation}
which shows that $f(x)\to\infty$ as $x\to\infty$. 
Similarly, $f(x)\to\infty$ as $x\to-\infty$. 
So, 
\begin{equation}\label{eq:to infty}
	f(x)\underset{|x|\to\infty}\longrightarrow\infty. 
\end{equation}
Therefore and in view of \eqref{eq:m in R}, introducing now 
\begin{equation*}
	L:=\{x\in\R\colon f(x)\le m+1\},\quad u:=\inf L, \quad\text{and}\quad v:=\sup L, 
\end{equation*}
one sees that $v<\infty$ and $u>-\infty$. 
Also, by the convexity of $f$ and the condition~%(II)(i)
$\card\dom f>1$, $L$ is a nonempty interval with the endpoints $u$ and $v$, and 
\begin{equation*}
	m=\inf\limits_{[u,v]}f. 
\end{equation*}

Moreover, $u<v$. Indeed, if $u=v$, then the nonempty interval $L$ must be the singleton set $\{u\}$. So, $f>m+1$ on $\R\setminus\{u\}$ and hence $f(u)=m$. Therefore (cf.\ \eqref{eq:f(x)ge}), for each $x\in(u,\infty)$ and all $\vp\in(0,x-u)$ one has 
$f(x)\ge f(u)+\frac{x-u}\vp\big(f(u+\vp)-f(u)\big)\ge f(u)+\frac{x-u}\vp\underset{\vp\downarrow0}\longrightarrow\infty$, so that $f=\infty$ on $(u,\infty)$; similarly, $f=\infty$ on $(-\infty,u)$, which contradicts the condition 
%(II)(i) of Theorem~\ref{th:} 
$\card\dom f>1$. 
Thus, 
\begin{equation*}
	-\infty<u<v<\infty\quad\text{and}\quad f(x)>m+1\text{ for all real } x\notin[u,v]. 
\end{equation*}
So, eventually $f_n(u-1)\wedge f_n(v+1)>m+1$, whereas $f(w)<m+1$ for some $w\in L\subseteq[u,v]$ and hence eventually $f_n(w)<m+1$. 
Now the convexity of $f_n$ implies (cf.\ \eqref{eq:f(x)ge}) that uniformly over all $x\in[v+1,\infty)$ with $f_n(x)<\infty$ eventually one has 
$f_n(x)\ge f_n(v+1)+\frac{x-v-1}{v+1-w}\big(f_n(v+1)-f_n(w)\big)\ge f_n(v+1)>m+1$, and so, 
$\inf\limits_{[v+1,\infty)}f_n\ge m+1$. Similarly, eventually $\inf\limits_{(\infty,u-1]}f_n\ge m+1$, whence  
\begin{equation}\label{eq:u-1,v+1}
	m_n=\inf\limits_{[u-1,v+1]}f_n. 
\end{equation} 

Introduce next   
\begin{equation}\label{eq:xmin,xmax}
	x_{\min}:=\inf\dom f\quad\text{and}\quad x_{\max}:=\sup\dom f;   
\end{equation}
by the condition $\card\dom f>1$, $\dom f$ is a nonempty interval with endpoints $x_{\min}$ and $x_{\max}$, so that 
\begin{equation}\label{eq:xmin<xmax}
	-\infty\le x_{\min}<x_{\max}\le\infty. 
\end{equation}
Thus, in view of \eqref{eq:m in R}, the convex function $f$ is real-valued on the interval $(x_{\min},x_{\max})$ and therefore has the left and right derivatives, $f'_-$ and $f'_+$, on $(x_{\min},x_{\max})$. 
Now one can also introduce  
\begin{equation*}
	x_+:=x_{\max}\wedge\inf E_+\quad\text{and}\quad x_-:=x_{\min}\vee\sup E_-, 
\end{equation*}
where 
\begin{equation*}
	E_+:=\{x\in(x_{\min},x_{\max})\colon f'_+(x)>0\} \quad\text{and}\quad 
	E_-:=\{x\in(x_{\min},x_{\max})\colon f'_-(x)<0\}. 
\end{equation*}
%as usual, we assume the conventions $\inf\emptyset:=\infty$ and $\sup\emptyset:=-\infty$. 
Note that for any $x\in E_-$ and $y\in E_+$ one has $x\le y$ -- because for any $x$ and $y$ such that $x_{\min}<y<x<x_{\max}$ one has $f'_+(y)\le f'_-(x)$, by the convexity of $f$.  
So, 
\begin{equation}\label{eq:xmin<x-<x+<xmax}
	x_{\min}\le x_-\le x_+\le x_{\max}.   
\end{equation}
Comparing this with the strict inequality in \eqref{eq:xmin<xmax}, one sees that $x_+>x_{\min}$ or $x_-<x_{\max}$. These latter two cases are quite similar to each other, and in fact they can be deduced from each other by the horizontal flip $\R\times[-\infty,\infty]\ni(x,\tau)\mapsto(-x,\tau)$ applied to the graphs of the functions $f_n$ and $f$. 

%\newpage 

So, w.l.o.g.\ 
\begin{equation}\label{eq:xmin<x_+}
	x_{\min}<x_+\le x_{\max}. 
\end{equation}

The right derivative, $f'_+$, of the convex function $f$ is nondecreasing on $(x_{\min},x_{\max})$. 
Therefore, the set $E_+$ is a (possibly empty) interval with endpoints $x_+$ and $x_{\max}$. 
So, $f'_+\le0$ and hence $f$ is nonincreasing on the interval $(x_{\min},x_+)$, and the latter interval is nonempty, in view of \eqref{eq:xmin<x_+}.  
Thus, there exists the left limit of the function $f$ at the point $x_+$:
\begin{equation}\label{eq:p}
	p:=f\big((x_+)-\!\big)\in[m,\infty).  
\end{equation}

Observe that 
\begin{equation}\label{eq:f>p}
	f>p\text{ in a right neighborhood of }x_+. 
\end{equation}
Indeed, if $x_+=x_{\max}$, then $f=\infty>p$ on the entire interval $(x_+,\infty)=(x_{\max},\infty)$. 
So, in view of \eqref{eq:xmin<x_+}, one may now assume that $x_+<x_{\max}$, so that $x_+$ is in the interval $(x_{\min},x_{\max})$, which latter is the interior of the set $\dom f$. 
Therefore and in view of \eqref{eq:m in R}, the convex function $f$ is continuous at $x_+$, with $f(x_+)=f\big((x_+)-\!\big)=p$. 
Because $f$ is strictly increasing on the interval $(x_+,x_{\max})$ and continuous at $x_+$, $f$ is strictly increasing on the interval $[x_+,x_{\max})$ as well, which implies that $f>f(x_+)=p$ on $(x_+,x_{\max})$. 
Thus, \eqref{eq:f>p} is verified. 

Take now any $\de\in(0,\infty)$. Then take any 
\begin{equation}\label{eq:x++}
	x_{++}\in(x_+,x_++\de) 
\end{equation}
such that 
\begin{equation}\label{eq:f(x++)>p}
	f(x_{++})>p; 
\end{equation}
such a point $x_{++}$ exists by \eqref{eq:f>p}.  
Next, take any 
\begin{equation}\label{eq:vp}
	\vp\in\big(0,\de\wedge1\wedge(x_+-x_{\min})\wedge\sqrt{f(x_{++})-p}\,\big), 
\end{equation}
which is possible in view of \eqref{eq:xmin<x_+} and \eqref{eq:f(x++)>p}. 
Then on, successively take  
\begin{align}
	y&\in\big(\max\big[\tfrac12(x_++x_{\min}),x_+-\vp\big],\;x_+-\tfrac\vp2\big), \label{eq:y}\\ 
	z&\in(y+\tfrac\vp2,\,x_+)\ \text{such that } f(z)<p+\vp^2, \label{eq:z} \\ 
	w&\in\big(x_{\min},\,y-\tfrac12(x_+-x_{\min})\big); \label{eq:w}    
\end{align} 
%such choices are possible for the following reasons: 
here, (i) the choice of $y$ is possible by the choice of $\vp$; (ii) the choice of $w$ is possible by the choice of $y$; and (iii) the choice of $z$ is possible by the choice of $y$, the condition $\vp\ne0$, and the definition of $p$ in \eqref{eq:p}. 
The conditions on $\vp,y,z,w$ listed in \eqref{eq:vp}, \eqref{eq:y}, \eqref{eq:z}, \eqref{eq:w} imply 
\begin{equation}\label{eq:w,y,z}
	\begin{gathered}
	x_{\min}<w<y<z<x_+, \quad 
	y>x_+-\vp, \quad 
%	0<\vp<1\wedge(x_+-y), 
%	\quad 
	z-y>\vp/2, 
	\quad y-w>\tfrac12\,(x_+-x_{\min}), \\ 
	\quad\text{and}\quad  
	f(x_{++})>p+\vp^2. 
%	, \\ \text{where } 
%	\vp:=1\wedge(x_+-y).  
\end{gathered}
\end{equation}

%*****************************

Since $f$ is nonincreasing on $(x_{\min},x_+)$, one has $f(y)\ge p$. This and the inequalities $f(z)<p+\vp^2$ in \eqref{eq:z} and $f(x_{++})>p+\vp^2$ in \eqref{eq:w,y,z} show that eventually 
\begin{equation}\label{eq:f_n >,< }
	f_n(y)\ge p-\vp^2\quad\text{and}\quad f_n(z)<p+\vp^2<f_n(x_{++}). 
\end{equation}
Also, the values $f(y)$ and $f(z)$ are real, in view of \eqref{eq:m in R} and because $\{y,z\}\subset\dom f$. 
So, eventually the values $f_n(y)$ and $f_n(z)$ are real. 
Now, by the convexity of $f_n$, 
\eqref{eq:f_n >,< }, \eqref{eq:w,y,z}, and \eqref{eq:p}, 
uniformly over all $x\in[u-1,y]$ one eventually has 
\begin{multline}\label{eq:f_n,x<y}
f_n(x)\ge
\frac{z-x}{z-y}\,f_n(y)-\frac{y-x}{z-y}\,f_n(z)
\ge p-\frac{z+y-2x}{z-y}\,\vp^2
\ge p-2(z+y-2x)\vp \\ 
\ge p-2\big(2x_+-2(u-1)\big)\vp\ge m-4\big(x_+-(u-1)\big)\vp, 
%\underset{\vp\downarrow0}\longrightarrow p\ge m, 	
\end{multline}
so that eventually 
\begin{equation}\label{eq:<y}
	\inf\limits_{[u-1,y]}f_n\ge m-c_1\vp\ge m-c_1\de, 
\end{equation}
where $c_1:=\max\big[0,4\big(x_+-(u-1)\big)\big]$; the last inequality in \eqref{eq:<y} holds because of \eqref{eq:vp}. 

Somewhat similarly, 
uniformly over all $x\in[y,x_{++}]$ one eventually has 
\begin{equation}\label{eq:f_n,xmax>x>y}
f_n(x)\ge
\frac{x-w}{y-w}\,f_n(y)-\frac{x-y}{y-w}\,f_n(w)
\ge p-\de-\frac{4\big(f(w)+1\big)}{x_+-x_{\min}}\,\de;  
\end{equation}
here we used the relations $\frac{x-w}{y-w}\ge1$, $p\ge m=1>\vp>0$, $f_n(y)\ge p-\vp^2\ge p-\vp=0\vee(p-\vp)$, $\vp<\de$, $0\le x-y\le\vp+\de<2\de$, $y-w>\tfrac12\,(x_+-x_{\min})$, and $0\le f(w)-m=f(w)-1<f_n(w)<f(w)+1$, which hold at least eventually. 
So, eventually 
\begin{equation}\label{eq:>y,<xmax}
	\inf\limits_{[y,x_{++}]}f_n\ge m-c_2\de, 
\end{equation}
where $c_2:=1+4\big(f(w)+1\big)/(x_+-x_{\min})$. 

Further, 
uniformly over all $x\in(x_{++},\infty)$ with $f_n(x)<\infty$ one eventually has 
\begin{equation}\label{eq:f_n,x>xmax}
f_n(x)\ge
f_n(x_{++})+\frac{x-x_{++}}{x_{++}-z}\big(f_n(x_{++})-f_n(z)\big)
\ge f_n(x_{++})>p+\vp^2\ge m;    
\end{equation}
the second inequality here holds by \eqref{eq:f_n >,< }. 
So, eventually 
\begin{equation}\label{eq:>xmax}
	\inf\limits_{(x_{++},\infty)}f_n\ge m.  
\end{equation}
Combining this with \eqref{eq:u-1,v+1}, \eqref{eq:<y}, and \eqref{eq:>y,<xmax}, one sees that, for any real $\de>0$, eventually 
\begin{equation}\label{eq:>m-c vp}
	m_n\ge m-c\de,     
\end{equation}
where $c:=c_1\vee c_2$. 
%Comparing this with \eqref{eq:limsup<m}, one concludes that indeed $m_n\to m$. 
%
Thus, 
one obtains a contradiction with \eqref{eq:m_infty}--\eqref{eq:limsup<m}, so that 
the implication (II)$\implies$(I) is verified. 

It remains to consider the reverse implication,  

(I)$\implies$(II).\quad Equivalently, let us verify the implication $\neg$(II)$\implies\neg$(I), where the symbol $\neg$ denotes the negation, as usual. 
Thus, let us assume that the condition (II) in Theorem~\ref{th:} fails to hold. We have then to show that $f$ is not infimum-stable. 

If $\card\dom f=0$, then $f=\infty$ on $\R$ and hence $m=\infty$. On the other hand, defining convex functions $f_n$ by the formula $f_n(x):=x+n$ for all natural $n$ and all $x\in\R$, one sees that $f_n\to f$ pointwise, whereas $m_n=-\infty\not\to\infty=m$. So, $f$ is not infimum-stable in the case $\card\dom f=0$. 

Next, if $\card\dom f=1$ and $m>-\infty$, then $\dom f$ is the singleton set $\{x_0\}$ for some $x_0\in\R$, and $m=f(x_0)\in\R$. So, in view of possible vertical and horizontal shifting, let us make w.l.o.g.\ the simplifying assumptions that $x_0=0$ and $f(x_0)=1$. 
Then $m=f(0)=1$ and $f=\infty$ on $\R\setminus\{0\}$. 
Let us now define convex functions $f_n$ by the formula $f_n(x):=|1+nx|$ 
%\begin{equation}
%	f_n(x):=
%	\left\{
%	\begin{alignedat}{2}
%	&\infty&&\text{ if } x\in(-\infty,0), \\ 
%	%&0&&\text{ if } x=0, \\ 
%	&-nx&&\text{ if } x\in[0,\tfrac1n], \\ 
%	&nx-2&&\text{ if } x\in[\tfrac1n,\infty),   
%	\end{alignedat}
%	\right.
%\end{equation}
for all natural $n$ and all $x\in\R$. 
Then $f_n\to f$ pointwise, whereas $m_n=0\not\to 1=m$. So, $f$ is not infimum-stable in the case $\card\dom f=1$ as well. 

By the assumption $\neg$(II), it remains to consider the case when 
$\card\dom f>1$, $f$ is monotonic, and $m>-\infty$. 
Then w.l.o.g.\ $f$ is nondecreasing (say), %, whence $f\big((-\infty)+\big)=m$. 
whence $f\big((-\infty)+\big)\le f(x)<\infty$ for any $x$ in the set $\dom f$, which is nonempty by the assumption 
$\card\dom f>1$, and so, 
$%\begin{equation}\label{eq:f(-infty)<infty}
	f\big((-\infty)+\big)<\infty. 
$ %\end{equation}
Therefore, defining convex functions $f_n$ by the formula $f_n(x):=f(x)+x/n$ for all natural $n$ and all $x\in\R$, one has 
$m_n\le f_n\big((-\infty)+\big)=f\big((-\infty)+\big)-\infty=-\infty$ for all $n$, which implies, in view of the assumption $m>-\infty$, that $m_n=-\infty\not\to m$, whereas   
$f_n\to f$ pointwise. 
%, whereas $m_n\le f_n\big((-\infty)+\big)=f\big((-\infty)+\big)-\infty=-\infty$, because
%$f\big((-\infty)+\big)\le f(x)<\infty$ for any $x$ in the set $\dom f$, which is nonempty by the assumption 
%$\card\dom f>1$. 
%Therefore and by the assumption 
%$m>-\infty$, one sees that 
%
It follows that  
$f$ is not infimum-stable in this remaining case as well. 

Thus, the proof of the implication  
(I)$\implies$(II) is completed, and so is the entire proof of Theorem~\ref{th:}. 
\end{proof}

%\newpage

\begin{proof}[Proof of Corollary~\ref{cor:fin,fin}]\ 
Consider first the implication 

(II)$\implies$(I).\quad Assume that the function $f\colon C\to[-\infty,\infty]$ is convex and satisfies the condition (II) of Corollary~\ref{cor:fin,fin}, and take any sequence $(f_n)$ of convex functions $f_n\colon C\to[-\infty,\infty]$ converging to $f$ pointwise. 
Extend the function $f$, defined on $C$, to the function $\tf$ as in \eqref{eq:tf}, 
%defined on $\R$ by the conditions $\tf:=f$ on $C$ and $\tf:=\infty$ on $\R\setminus C$, 
and similarly extend $f_n$ to $\tf_n$, for each $n$. 
Then, on the entire real line $\R$, the functions $\tf$ and $\tf_n$ are convex, and $\tf_n\to\tf$ pointwise. Moreover, the condition (II) of Corollary~\ref{cor:fin,fin} implies that the condition (II) of Theorem~\ref{th:} is satisfied with the extended function $\tf$ in place of $f$ there; indeed, if $-\infty<c_{\min}<c_{\max}<\infty$ and $\card\dom f>1$, then $\tf$ is not monotonic. 
So, by Theorem~\ref{th:}, 
$\inf\limits_C f_n=\inf\limits_\R\tf_n\longrightarrow\inf\limits_\R\tf=\inf\limits_C f$. 
Thus,  the implication (II)$\implies$(I) in Corollary~\ref{cor:fin,fin} is verified. 

Consider now the reverse implication,  

(I)$\implies$(II).\quad Equivalently, let us verify the implication $\neg$(II)$\implies\neg$(I). 
Thus, let us assume that the condition (II) in Corollary~\ref{cor:fin,fin} fails to hold. We have then to show that $f$ is not infimum-stable. 

If $\card\dom f=0$, then $f=\infty$ on $C$. 
By horizontal shifting and the condition \eqref{eq:cmin<cmax}, w.l.o.g.\ the point $0$ is in the interior %, say $\inter C$, 
of the set $C$. 
Defining now convex functions $f_n$ by the formula $f_n(x):=%\break 
n|1+nx|$ for all natural $n$ and all $x\in C$, one sees that $f_n\to f$ pointwise, whereas eventually one has $-\frac1n\in C$ and hence $\inf\limits_C f_n=0\not\to\infty=\inf\limits_C f$. So, $f$ is not infimum-stable in the case $\card\dom f=0$. 

It remains to consider the case when $\card\dom f=1$ and $\inf\limits_C f>-\infty$. Then $\dom f$ is the singleton set $\{x_0\}$ for some $x_0\in C$, and $\inf\limits_C f=f(x_0)\in\R$. 
By \eqref{eq:cmin<cmax}, either $x_0>c_{\min}$ or $x_0<c_{\max}$. 
These two cases are quite similar to each other. So, assume w.l.o.g.\ that $x_0>c_{\min}$. 
Moreover, by vertical and/or horizontal shifting, w.l.o.g.\ $x_0=0$ and $f(x_0)=1$. 
So, $f(0)=1$, $f=\infty$ on $C\setminus\{0\}$, and $\inf C=c_{\min}<0\in C$. 
Let us now define convex functions $f_n$ by the formula $f_n(x):=|1+nx|$ 
for all natural $n$ and all $x\in C$. 
Then $f_n\to f$ pointwise, whereas 
eventually $\inf\limits_C f_n=0\not\to1=\inf\limits_C f$. 
So, $f$ is not infimum-stable whenever the condition (II) in Corollary~\ref{cor:fin,fin} fails to hold. 
\end{proof}

\begin{proof}[Proof of Corollary~\ref{cor:-inf,fin}]\ 
The proof of the implication 
(II)$\implies$(I) of this corollary repeats the corresponding part of the proof of Corollary~\ref{cor:fin,fin} almost literally.  
The main difference is that here --- instead of the conditions $-\infty<c_{\min}<c_{\max}<\infty$ and $\card\dom f>1$ --- one should use the conditions  $-\infty=c_{\min}<c_{\max}<\infty$ and $\card\dom f>1$ to find that the function $\tf$ is not nonincreasing and then conclude that $\tf$ is not monotonic (given that $f$ is not nondecreasing and hence $\tf$ is not nondecreasing). 
%
%, and of $f$ being not nondecreasing to conclude that the function $\tf$ is not nonincreasing and hence is not monotonic. 

As for the reverse implication, (I)$\implies$(II), in Corollary~\ref{cor:-inf,fin}, its proof is almost literally the same as the corresponding part of the proof of Theorem~\ref{th:}. 
The differences are few and small: ``Theorem~\ref{th:}''; ``monotonic''; $\R$ as the domain of $f$ and $f_n$; $m$; and $m_n$ have to be replaced, respectively, by ``Corollary~\ref{cor:-inf,fin}''; ``nondecreasing''; $C$; $\inf\limits_C f$; and $\inf\limits_C f_n$. 
\end{proof}

\begin{proof}[Proof of Corollary~\ref{cor:fin,inf}]\ 
This proof is quite similar to that of Corollary~\ref{cor:-inf,fin}. 
Alternatively, Corollary~\ref{cor:fin,inf} can be obtained immediately from Corollary~\ref{cor:-inf,fin} by horizontal flipping. 
\end{proof}

%\newpage
%
%\begin{cor2a}\label{cor:2a}
%Suppose that $-\infty<c_{\min}<c_{\max}<\infty$, the function $f$ is convex, and $f>-\infty$ on $C$. Then $f$ is infimum-stable if and only if $\card\dom f>1$.  
%\end{cor2a}

\begin{proof}[Proof of Corollary~2a]\ 
Suppose that indeed $-\infty<c_{\min}<c_{\max}<\infty$, the function $f$ is convex, and $f>-\infty$ on $C$. 
In view of Corollary~\ref{cor:fin,fin}, it is enough to show that then $\inf\limits_C f>-\infty$. 
This conclusion is obvious if $\card\dom f=0$. If $\card\dom f=1$, then the condition $f>-\infty$ on $C$ implies that there is a point $x_0\in C$ such that $f(x_0)\in\R$ and $f=\infty$ on $C\setminus\{x_0\}$, so that again the conclusion $\inf\limits_C f>-\infty$ follows. 

Finally, if $\card\dom f>1$, then there is a point $x_0\in C$ such that $f$ is finite in a neighborhood of $x_0$. So, by the convexity of $f$, there is a finite right derivative, $f'_+(x_0)$, of $f$ at $x_0$. So, $f(x)\ge f(x_0)+f'_+(x_0)(x-x_0)$ for all $x\in C$ and therefore  
$\inf\limits_C f\ge\inf\limits_{x\in C}[f(x_0)+f'_+(x_0)(x-x_0)]>-\infty$, since the set $C$ is bounded. Thus, in all cases the conclusion $\inf\limits_C f>-\infty$ holds. 
\end{proof}

%!!!!!!!!!!!! {noinfoline}

\bibliographystyle{abbrv}
%\bibliographystyle{ims}
%\bibliography{are.citations}
%\bibliography{citat}

%\bibliography{citations}

\bibliography{C:/Users/Iosif/Dropbox/mtu/bib_files/citations12.13.12}
%\bibliography{C:/Users/iosif-home-2011/Dropbox/mtu/bib_files/citations12.13.12}
%\bibliography{C:/Users/Iosif/Documents/mtu_home01-30-10/bib_files/citations}

\end{document}